\documentclass[10pt]{amsart}
\usepackage{eurosym}
\usepackage{amsmath,amssymb,amsfonts,amscd,xcolor,hyperref,graphicx,ulem}
\usepackage[left=2.75cm,top=3cm,right=2.75cm,bottom=3cm]{geometry}
\usepackage[initials]{amsrefs}
\newtheorem{theorem}{Theorem}[section]
\newtheorem{definition}[theorem]{Definition}

\newtheorem{lemma}[theorem]{Lemma}

\newtheorem{example}[theorem]{Example}

\theoremstyle{definition}
\newtheorem{remark}[theorem]{Remark}

\numberwithin{equation}{section}

\newcommand{\N}{\mathbb{N}}

\renewcommand{\epsilon}{\varepsilon}
\renewcommand{\phi}{\varphi}
\renewcommand{\emptyset}{\varnothing}

\begin{document}
\title[Algebrability and Riemann integrability of the composite function]{Algebrability and Riemann integrability\\ of the composite function}
\author[D'Aniello]{E. D'Aniello}
\address[E. D'Aniello]{\mbox{}\newline \indent  Dipartimento di Matematica e Fisica, \newline\indent Universit\`a degli Studi della Campania ``Luigi Vanvitelli'', \newline\indent  Viale Lincoln n. 5,\newline\indent  81100 Caserta, Italy.}
\email{emma.daniello@unicampania.it}

\author[Fern\'andez]{J. Fern\'andez-S\'anchez}
\address[J. Fern\'andez-S\'anchez]{\mbox{}\newline \indent
	Instituto de Matem\'atica Interdisciplinar (IMI), \newline \indent
	Universidad Complutense de Madrid, \newline \indent
	Madrid (Spain), \newline \indent
	\sc{and}\newline \indent
	Grupo de investigaci\'on de Teor\'ia de C\'opulas y Aplicaciones, \newline \indent
	Universidad de Almer\'ia, \newline \indent
	Carretera de Sacramento s/n, \newline \indent
	04120 Almer\'ia (Spain).}
\email{juanfernandez@ual.es}

\author[Maiuriello]{M. Maiuriello}
\address[M. Maiuriello]{\mbox{}\newline \indent  Dipartimento di Matematica e Fisica, \newline\indent Universit\`a degli Studi della Campania ``Luigi Vanvitelli'', \newline\indent  Viale Lincoln n. 5,\newline\indent  81100 Caserta, Italy.}
\email{martina.maiuriello@unicampania.it}

\author[Seoane]{J.B. Seoane--Sep\'ulveda}
\address[J.B. Seoane--Sep\'ulveda]{\mbox{}\newline\indent Instituto de Matem\'atica Interdisciplinar (IMI), \newline\indent Departamento de An\'{a}lisis Matem\'{a}tico y Matem\'atica Aplicada,\newline\indent Facultad de Ciencias Matem\'aticas, \newline\indent Plaza de Ciencias 3, \newline\indent Universidad Complutense de Madrid,\newline\indent 28040 Madrid, Spain.}
\email{jseoane@ucm.es}

\begin{abstract}
In this note we show that there exist a $2^\mathfrak{c}$-generated free algebra $\mathcal{S} \subset \mathbb{R}^\mathbb{R}$ of Riemann integrable functions and a free algebra $\mathcal{C} \subset \mathbb{R}^{[0,1]}$ of continuous functions, having $\mathfrak{c}$-generators,  such that $r \circ c$ is not Riemann integrable for any $r \in \mathcal{S}$ and $c \in \mathcal{C}$. This result is the best possible one in terms of lineability within these families of functions and, at the same time, an improvement of a precious result (\cite[Theorem 2.7]{A}). In order to achieve our results we shall employ set theoretical tools such as the Fichtenholz-Kantorovich-Hausdorff theorem, Cantor-Smith-Volterra--type sets, and classical real analysis techniques. 
\end{abstract}

\keywords{lineability, algebrability, bounded real function.}
\subjclass[2020]{15A03, 46B87.}
\thanks{}
\maketitle

\section{Introduction and Preliminaries}

The search for linear structures of mathematical objects enjoying certain \textit{special} property has, for the past decade, become a sort of a trend in many  different areas of Mathematics and, as a consequence of this, a vast literature on this topic has recently been built, from Linear Chaos to Real and Complex Analysis, passing through Set Theory and Linear and Multilinear Algebra, Operator Theory, Topology and Measure Theory, Functional Analysis, or Abstract Algebra \cite{A,book,BAMS2014,BAMS2019}. 

Let us recall some terminology we shall need throughout this work (which can be found in, for instance, \cite{BAMS2014,book}). Assume that $X$ is a vector space and $\alpha$ is a cardinal number. Then a subset \,$A \subset X$ \,is said to be:
\begin{enumerate}
	\item[$\bullet$] {\it lineable} if there is an infinite dimensional vector space $M$ such that $M \setminus \{0\}\subset A$;
	\item[$\bullet$] {\it $\alpha$-lineable} if there exists a vector space $M$ with dim$(M) = \alpha$ and $M \setminus \{0\} \subset A$.
\end{enumerate}

And, provided that $X$ is a vector space contained in some (linear) algebra, then \,$A$ \,is called:
\begin{enumerate}
	\item[$\bullet$] {\it algebrable} if there is an algebra \,$M$ so
	that $M \setminus \{0\} \subset A$ and $M$ is infinitely generated, that is, the cardinality of any system of generators of $M$ is infinite;
	\item[$\bullet$] {\it strongly $\alpha$-algebrable} if there exists an $\alpha$-generated {\it free} algebra $M$ with $M \setminus \{0\} \subset A$.
	Recall that if $X$ is contained in a commutative algebra, then a set $B \subset X$ is a \textit{generating set} of some free algebra contained in \,$A$ \,if and only if for any $N \in \N$, any nonzero polynomial $P$ in \,$N$ variables without constant term and any distinct $f_1, \dots ,f_N \in B$, we have \,$P(f_1, \dots ,f_N) \in A\setminus\{0\}$. 
\end{enumerate}

Of course, any form of (infinite) algebrability implies lineability. We shall denote  by $\mathcal{B}$ the subset of  $ \mathbb{R}^{[0,1]}$ of  bounded functions.  This paper focuses on a very particular class of real valued functions: the bounded ones. Of course, the set $\mathcal{B}$ defined above is, in itself, a vector space and, simultaneously, an algebra as well. However, the question that arises is: How large (in terms of dimension) can this vector space get? And regarding seeing it as an algebra... How large (in terms of cardinality) can the set of free generators be for this algebra? Of course, and due to cardinality issues, the dimension of this vector space (and the cardinality of the set of generators for the potential free algebra) cannot exceed $2^\mathfrak{c}$.  

%

When dealing with lineability, one can encounter many different situations. Just to cite some illustrative (and now classical) examples, for instance:

\begin{example}\label{ex1}
	A function $f  \in \mathbb{R}^\mathbb{R}$ is said to be everywhere surjective (ES, for short) if  $f(I) = \mathbb{R}$ for every non-void open interval $I$ of $\mathbb{R}$. In \cite{arongurariyseoane2005} the authors showed that ES is, indeed, $2^\mathfrak{c}$-lineable.
\end{example}

\begin{example}
	Let us denote by $\widehat{C}(\mathbb{R})$ the subset of  $C(\mathbb{R})$ of continuous real-valued functions attaining their absolute maximum exactly once. In \cite{GurProb} it was proved that $\widehat{C}(\mathbb{R})$  is $2-$lineable but not $3-$lineable.
\end{example}

\begin{example}
In 1923, Sierpi\'nski and Zygmund \cite{sierpinskizygmund1923}  proved the existence of  a function \(f\colon\mathbb R\to\mathbb R\) such that, for any set \(Z\subset\mathbb R\) of cardinality $\mathfrak{c}$, the restriction \(f\vert_Z\) is not a Borel map {\rm (}and, in particular, not continuous{\rm )}. A  function \(f\colon\mathbb R\to\mathbb R\) is a Sierpi\'nski-Zygmund function if it sa\-tis\-fies the condition in Sierpi\'nski-Zygmund's Theorem. In 2012 \cite{gamezseoane2012}, the authors showed that the $2^{\mathfrak{c}}$-lineability of the set of Sierpi\'{n}ski-Zygmund functions is, actually, undecidable.
\end{example}

These previous examples illustrate that anything can actually occur and that intuition does not really help here. 

Our goal in this note is to improve \cite[Theorem 2.7]{A} by showing that there exist a $2^\mathfrak{c}$-generated free algebra $\mathcal{S} \subset \mathbb{R}^\mathbb{R}$ of Riemann integrable functions and a free algebra $\mathcal{C} \subset \mathbb{R}^{[0,1]}$ of continuous functions, having $\mathfrak{c}$-generators,  such that $r \circ c$ is not Riemann integrable for any $r \in \mathcal{S}$ and $c \in \mathcal{C}$.  In order to achieve this we shall need to make a smart use of several set theoretical techniques (such as the Fichtenholz-Kantorovich-Hausdorff  theorem) and classical real analysis tools as well. The notation shall be rather usual.

\section{The main result}

With the goal of being a self-contained paper, let just provide some necessary definitions and tools in order to provide a straight proof of Lemma \ref{mainthm} below (that originally appeared proved in \cite[Theorem 2.2]{ppp} as a consequence of a more intricate result). 

As it is common in certain grounds, if $X$ stands for any nonempty set, for any set $A\subseteq X$ we denote $A^{0}=X\setminus A$ and $A^{1}=A$. Let us recall the following definition (see, e.g., \cite{J}).

\begin{definition}
	Let $X$ be a nonempty set. A family $\mathcal{A}$ of subsets of $X$ is
	independent provided that for any finitely many distinct elements $%
	A_1,\ldots,A_n $ of $\mathcal{A}$ we have that $A_1^{\varepsilon_1}\cap%
	\cdots\cap A_n^{\varepsilon_n}\neq\emptyset$ for any $\varepsilon_1,\ldots,%
	\varepsilon_n\in\{0,1\}$.
\end{definition}

A classical set theoretical result, known as the Fichtenholz-Kantorovich-Hausdorff theorem (\cites{fich,h,J}) is also needed to achieve our goal.

\begin{theorem}[Fichtenholz-Kantorovich-Hausdorff]\label{FKH}
For any set $X$ of infinite cardinality there exists an independent family $\mathcal{A}\subseteq \mathcal{P}(X)$ of cardinality $2^{\text{card}(X)}$.
\end{theorem}


Recall that $\mathcal{P}(X)$ denotes the power set of $X$. 


From Theorem \ref{FKH}, we have that  given any set $X$ of cardinality $\kappa \geq \aleph _{0}$, there is an independent family $\mathcal{A}$ of $2^{\kappa }$-many subsets of $X$. Moreover, we have that $A_{1}^{\varepsilon _{1}}\cap \cdots \cap A_{n}^{\varepsilon _{n}}$ is infinite for every $A_{1},\ldots ,A_{n}\in  \mathcal{A}$, $\varepsilon _{1},\ldots ,\varepsilon _{n}\in \{0,1\}$ and $n\in \mathbb{N}$. Indeed, by way of contradiction assume that $$\text{card}(A_{1}^{\varepsilon _{1}}\cap \cdots \cap A_{n}^{\varepsilon _{n}})=k\in \mathbb{N}\setminus \{0\},$$ 
\text{i.e.},
$$A_{1}^{\varepsilon _{1}}\cap \cdots \cap A_{n}^{\varepsilon _{n}}=\{x_{1},\ldots ,x_{k}\}.$$ 
For every $i\in \{1,\ldots ,k\}$, take $A_{x_{i}}\in \mathcal{A}$ and $\varepsilon_{x_{i}}\in \{0,1\}$ such that $x_{i}\notin A_{x_{i}}^{\varepsilon _{x_{i}}}$ (with $A_{x_i}$ not belonging to $\{A_1, \ldots , A_n\}$), then 
\begin{equation*}
	A_{1}^{\varepsilon _{1}}\cap \cdots \cap A_{n}^{\varepsilon _{n}}\cap
	A_{x_{1}}^{\varepsilon _{x_{1}}}\cap \cdots \cap A_{x_{k}}^{\varepsilon
		_{x_{k}}}=\emptyset .
\end{equation*}

Let us recall that, if $\mathcal{H}$ is a Hamel basis of $\mathbb{R}$ as a  $\mathbb{Q}-$vector space,  \linebreak 
$h_{1},\ldots ,h_{n},h_{n+1}\in \mathcal{H}$ and $c_{1},\ldots ,c_{n}\in 
\mathbb{Q}$, the set of points $x\in \mathbb{R}$  of the form $$c_{1}h_{1}+\ldots +c_{n}h_{n}+ch_{n+1}$$ with $c\in \mathbb{Q}$, is dense in $\mathbb{R}$. Also, given $\mathcal{H}$ a Hamel basis as before, and if  $A$  is a subset of $\mathcal{H}$, we shall denote by $f_{A}$ the $\mathbb{Q}-$linear function in $\mathbb{R}^\mathbb{R}$  such that, if $h \in\mathcal{H}$ it is
\begin{equation*}
	f_{A}(h)=\left\{ 
	\begin{array}{ll}
		1 & \text{if }h\in A, \\ 
		0 & \text{otherwise},%
	\end{array}%
	\right. 
\end{equation*}%
and $g_{A}$  is the function given by

\begin{equation}
	g_{A}(x):=\frac{f_{A}(x)}{1+\left\vert f_{A}(x)\right\vert }.  \label{g}
\end{equation}


Now, with the previous tools and notations at hand, we are able to state and prove the following tool we need that, as we mentioned above, it was already shown in \cite{ppp}. However, the technique in our proof is new and more direct.

\begin{lemma}\label{mainthm}
The subset $\mathcal{B}$ of  $ \mathbb{R}^{[0,1]}$ of  bounded functions is strongly $2^{\mathfrak{c}}-$algebrable.
\end{lemma}

\begin{proof}
Let us take $\mathcal{H}$ a Hamel basis such that $1 \in  \mathcal{H}$. If $\mathcal{A}$ is a family of independent sets in $\mathcal{H}\backslash \{1\},$ we shall see that the functions $g_{A}$ (defined as above, \eqref{g}) are algebraically independent.  

If $g_{A_1}, \ldots, g_{A_n}$ were algebraically dependent, there would exists a polynomial (in $n$ variables) $P \neq 0$ (and without constant term) in such a way that 
$P\left( g_{A_{1}}(r),\ldots, g_{A_{n}}(r)\right) = 0$, $\forall r \in [0,1]$. Since $P\neq 0$ there exist $x_{1},\ldots ,x_{n}\in [0,1]$, such that $P\left( x_{1},\ldots ,x_{n}\right) \neq 0.$ Next, take $A_{1},\ldots , A_{n}\in \mathcal{A}$ and let 
	\begin{eqnarray*}
		h_{1} &\in &A_{1}^{1}\cap A_{2}^{0}\cap \ldots \cap A_{n}^{0} \\
		h_{2} &\in &A_{1}^{0}\cap A_{2}^{1}\cap A_{3}^{0}\cap \ldots \cap A_{n}^{0}\\
		\cdots & \qquad &\cdots  \qquad  \cdots  \qquad \cdots \\
		h_{n} &\in &A_{1}^{0}\cap \ldots \cap A_{n-1}^{0}\cap A_{n}^{1} \\
		h_{n+1} &\in &A_{1}^{0}\cap \ldots \cap A_{n-1}^{0}\cap A_{n}^{0}.
	\end{eqnarray*}
	
If $r=c_{1}h_{1}+\ldots +c_{n}h_{n}+ch_{n+1},$we have that $f_{A_{i}}(r)=c_{i} \in \mathbb{Q}$ for every $i \in \{1, 2, \ldots, n\}$.  Once we fix  $\delta >0,$ we can find appropriate $c_{i}$'s in such a way that  $$\left\vert x_{i}-\frac{c_{i}}{1+c_{i}}\right\vert <\delta.$$

Thus, we can also find values of  $c_{i}$'s and $c$  in such a way that:
	\begin{equation*}
		\left\vert P\left( x_{1},\ldots ,x_{n}\right) -P\left( g_{A_{1}}(r),\ldots
		, g_{A_{n}}(r)\right) \right\vert <\frac{\left\vert P\left( x_{1},\ldots
			,x_{n}\right) \right\vert }{2}
	\end{equation*}
and $r\in [0,1]$. This is a contradiction, since we assumed that $P\left( g_{A_{1}}(r),\ldots, g_{A_{n}}(r)\right) = 0$, $\forall r \in [0,1]$. That is, we have algebraic independence. We let $\mathcal{G}$ stands for this free algebra (the algebra generated by the functions $g_A$ with $A \in \mathcal{A}$).
\end{proof}

The following remark is also an important tool, although it is, by itself, of independent interest.

\begin{remark}
Notice that, if  $r^{\prime }=c_{1}h_{1}+\ldots
+c_{n}h_{n}$  then, in every interval $]\alpha ,\beta \lbrack \subset
\lbrack 0,1]$, by choosing $c$ appropriately, there exists $r\in ]\alpha ,\beta \lbrack $ such that $$P\left(
g_{A_{1}}(r),\ldots , g_{A_{n}}(r)\right) =P\left( g_{A_{1}}(r^{\prime
}),\ldots , g_{A_{n}}(r^{\prime })\right) .$$
\end{remark}

Now, let us show the main result in this short note. Namely, the following theorem improves a result from \cite{A}.  More precisely, \cite[Theorem 2.7]{A}, regarding the algebraic genericity of the set of real functions that are not Riemann integrable (the result from \cite{A} is an immediate consequence of the following one). More particularly, in \cite[Theorem 2.7]{A} the authors constructed a $2^\mathfrak{c}$-dimensional space $V$ and a $\mathfrak{c}$-dimensional space $W$ of, respectively, Riemann integrable functions and continuous functions such that, for every $f\in V\setminus\{0\}$ and $g\in W\setminus\{0\}$, $f\circ g$ is not Riemann integrable. Here we improve this previous result in the best possible way (in terms of lineability), by showing that the previous $V$ and $W$ can be chosen to be free algebras both of the maximal possible dimension and size of the sets of generators. Let us show it.

\begin{theorem}
There exist 
\begin{itemize}
\item[(i.)] a $2^\mathfrak{c}$-generated free algebra $\mathcal{S} \subset \mathbb{R}^\mathbb{R}$ such that the restriction of any of its elements to any interval $[\alpha,\beta]$ is Riemann integrable and
\item[(ii.)]  a free algebra $\mathcal{C} \subset \mathbb{R}^{[0,1]}$ of continuous functions, having $\mathfrak{c}$-generators, 
\end{itemize}
such that $r \circ c$ is not Riemann integrable for any $r \in \mathcal{S}$ and $c \in \mathcal{C}$.
\end{theorem}

\begin{proof}	
Take a weighted Cantor set $T\subset \lbrack 0,1]$ (also known as Cantor-Smith-Volterra set), symmetric with respect to $1/2$. The set $[0,1]\setminus T$ is formed by the disjoint intervals $]a_{s},b_{s}[.$ Define the (continuous) function 
	\begin{equation*}
		g(x):=\left\{ 
		\begin{array}{cc}
			\min \{x-a_{s},b_{s}-x\}/s & \text{if }x\in ]a_{s},b_{s}[\text{ for some }s \in \N,
			\\ 
			0 & \text{otherwise}.%
		\end{array}%
		\right.
	\end{equation*}
	
	Denote by $\mathcal{C}$ the algebra generated by the powers $g^{h}$ with $h\in \mathcal{H}\subset \lbrack 1,2].$ It is a free algebra of continuous functions with $\mathfrak{c}$ generators. Next, consider $t:C\rightarrow
	\lbrack 0,1]$ the restriction of the Cantor function to the classical Cantor set $C$. Also, $h_{n}:[0,1[\rightarrow \lbrack \frac{1}{n+1%
	},\frac{1}{n}[$ is the increasing linear bijection and denote $%
	C_{n}=h_{n}(C\backslash \{1\})$ with $n\geq 1.$ 
	
	Now, for every $m$ belonging to the free algebra $\mathcal{G}$ constructed
	in the proof of Lemma \ref{mainthm}, let us define $r_{m}^{\ast
	}:[0,1]\rightarrow \mathbb{R}$ as 
	\begin{equation*}
		r_{m}^{\ast }(x)=\left\{ 
		\begin{array}{cc}
			m\left( t\left( h_{n}^{-1}(x)\right) \right)  & \text{if }x\in C_{n}\text{
				for some }n\geq 1, \\ 
			0 & \text{otherwise,}%
		\end{array}%
		\right. 
	\end{equation*}%
	which is well defined. Now, we let $r_{m}$ be the periodic extension of $%
	r_{m}^{\ast }.$ These functions are zero almost everywhere and, thus, they are locally $R$-integrable.  With all the previous elements at hand, denote by $\mathcal{S}$ the set of
	functions $r_{m}$ with (as we said earlier) $m$ belonging to the free
	algebra $\mathcal{G}$ previously constructed (see the proof of Lemma \ref%
	{mainthm}). The set $\mathcal{S}$ is a free algebra with $2^{\mathfrak{c}}$
	generators and the restriction to any interval $[\alpha ,\beta ]$ of any of
	its elements is Riemann integrable. If $c\in \mathcal{C}$ and $r\in \mathcal{S},$ we have that $r\circ c$ is not Riemann integrable.
	
	Take $x^{\prime }\in T$, then any neighborhood  $U$ of $x^{\prime }$
	contains infinitely many intervals of the form $]a_{s},b_{s}[$ of $%
	]0,1[\backslash T$ and with $a_{s},b_{s}\in T.$ 
	
	If $c$ is not the identically $0$ function, we have that  $c\left( \left[ a_{s},b_{s}\right] \right) $
	is a non-degenerated interval containing $0.$ Therefore, from some $n_{0}$ on, we have $c\left( \left[ a_{s},b_{s}\right] \right) $
	 contains one of the following sets: 
	 $$\bigcup _{n>n_{0}}C_{n} \, \text{ or } \, \bigcup
	_{n>n_{0}}-C_{n}.$$ 
	
	Recall that the reason why, previously, we introduced the functions $h_n$'s was to, actually, guarantee this previous fact. Without loss of generality, and for simplicity, suppose that it is actually the set  $\cup _{n>n_{0}}C_{n}.$ We have that  $t\left(
	h_{n}^{-1}(C_{n})\right) =[0,1[.$ For $m\in \mathcal{G}\backslash \mathcal{%
		\{}0\}$, there exists $\gamma _{m}\in \lbrack 0,1[$ such that $m\left( \gamma
	_{m}\right) \neq 0.$ That is, once we fix $m$, for every  neighborhood $U$ of $x'$ there exists $x_{U}$ such that $$r_{m}\left( c(x_{U})\right) =m\left( \gamma
	_{m}\right) \neq 0.$$ 
	
	Since $r_{m}\left( c(x^{\prime })\right) =r_{m}(0)=0,$ we have that  $r_{m}\circ c$ is not continuous at $x^{\prime },$ that is, 
	$r_{m}\circ c$ is not continuous in $T$ and, thus, it is not Riemann
	integrable, since $\lambda (T)>0$ (Lebesgue's theorem).
\end{proof}

\end{document}